\documentclass[12pt,reqno]{amsart}
\usepackage{amssymb, amsmath, amsthm}
\usepackage{verbatim}
\usepackage[colorlinks=true, urlcolor=blue, citecolor=blue,  linkcolor=blue]{hyperref}
\usepackage{fullpage}

\newtheorem{lem}{Lemma}[section]

\newtheorem{prop}[lem]{Proposition}
\newtheorem{cor}[lem]{Corollary}
\newtheorem{conj}[lem]{Conjecture}

\newtheorem{claim*}{Claim}
\newtheorem{thm*}{Theorem}
\newtheorem{thm}[lem]{Theorem}

\theoremstyle{definition}
\newtheorem{rmk}[lem]{Remark}
\newtheorem{defn}[lem]{Definition}

\newtheorem{question}[lem]{Question}

\numberwithin{equation}{section}
\numberwithin{table}{section}

\def\phi{\varphi}

\def\p{\mathfrak{p}}

\def\vp{\varphi}

\def\a{\alpha}

\def\P1{\mathbb{P}^1_F}

\def\P{\mathcal{P}}

\newcommand{\FF}{\mathbb{F}}
\newcommand{\PP}{\mathbb{P}}
\newcommand{\oo}{\mathfrak{o}}
\newcommand{\pp}{\mathfrak{p}}
\newcommand{\N}{\mathrm{N}}

\newcommand{\PGL}{\mathrm{PGL}}

\title{Newton's method over global height fields}

\author{Xander Faber}
\address{
Department of Mathematics \\
University of Hawaii \\
Honolulu, HI}
\email{xander@math.hawaii.edu}

\author{Adam Towsley}
\address{
Department of Mathematics\\
CUNY Graduate Center\\
New York, NY
}
\email{atowsley@gc.cuny.edu}

\begin{document}

\begin{abstract}
Newton's method is used to approximate roots of complex valued functions $f$ by creating a sequence of points that converges to some root of $f$ in the usual topology.
For any field $K$ equipped with a set of pairwise inequivalent absolute values satisfying a product formula, we completely describe the conditions under which Newton's method applied to a squarefree polynomial $f \in K \left[ x \right]$ will succeed in finding some root of $f$ in the $v$-adic topology for infinitely many places $v$ of $K$. Furthermore, we show that if $K$ is a finite extension of the rationals or of the rational function field over a finite field, then the Newton approximation sequence fails to converge $v$-adically for a positive density of places $v$.
\end{abstract}

\maketitle
\section{Introduction}

	Newton's method is one of the most efficient techniques for locating and approximating zeros of complex analytic or $p$-adic analytic functions in a single variable. Recall how it works: if $f(x)$ is an analytic function with a simple root $\alpha$, then for any $x_0$ sufficiently close to $\alpha$ we have 
	\[
		0 = f(\alpha) \approx f(x_0) + f'(x_0) (\alpha - x_0) \ \Longrightarrow \ \alpha \approx x_0 - \frac{f(x_0)}{f'(x_0)}. 
	\]
Setting $N_f(x) = x - \frac{f(x)}{f'(x)}$, 
one can prove that the orbit $(x_n)$ defined by $x_n = N_f(x_{n-1})$ for $n \geq 1$ converges quadratically to the root $\alpha$. 


	Now let $K$ be a global height field --- that is, a field equipped with a set $M_K$ of pairwise inequivalent absolute values $\{|\cdot|_v: v \in M_K\}$ satisfying an adelic property and a product formula.  The basic examples are number fields and function fields of transcendence degree~1. (See Section~\ref{Sec: Height fields} for more details.) Each $v$ induces a different metric structure on $K$. Fix a univariate polynomial function $f$ and an initial guess $x_0$, both defined over $K$. Define a sequence recursively by $x_n = N_f(x_{n-1})$ for $n \geq 1$, where as above we set 
		\[
			N_f(x) = x - f(x) / f'(x).
		\] 
We will say that $(x_n)$ is the \textbf{Newton approximation sequence} associated to $f$ and $x_0$ if it is not eventually periodic. Following an investigation of the adelic convergence of orbits of more general dynamical systems \cite{Silverman_Voloch}, Silverman and Voloch  asked the following question. (We paraphrase. See also Remark~\ref{Rmk: Silverman/Voloch} below.)
	
	\begin{quote}What can be said about the set of places $v$ of $K$ for which the Newton approximation sequence $(x_n)$ converges to some root of~$f$? \end{quote}
	
	Silverman and Voloch showed that the \textit{complement} of the set of places of convergence is infinite. Later, the first author and Voloch showed that the set of places of convergence itself is infinite, unless there is a specific kind of dynamical obstruction. Theorem~\ref{Thm: FV Main} summarizes these conclusions in the case of squarefree polynomials. Recall that if $\phi \in K(z)$ is a rational function of degree at least~2, a point $\alpha \in \PP^1(K)$ is called an exceptional fixed point for $\phi$ if $\phi(\alpha) = \alpha$ and $\phi$ is totally ramified at $\alpha$. 
	
\begin{thm}[\cite{FV, Silverman_Voloch}]
\label{Thm: FV Main}
	Let $K$ be a number field, and let $(x_n)$ be the Newton approximation sequence associated to a squarefree polynomial $f \in K[x]$ of degree at least~2 and an initial point $x_0 \in K$. 
	\begin{enumerate}
		\item  The sequence $(x_n)$ converges $v$-adically to the root $\alpha$ of $f$ for infinitely
		many places $v$ if and only if $\alpha$ is not an exceptional fixed point for the 
		dynamical system $N_f$. In particular, if $f$ is not quadratic, then 
		$(x_n)$ converges $v$-adically to some root of $f$ for infinitely many places $v$. 
		\item The sequence $(x_n)$ fails to converge in $\PP^1(K_v)$ for infinitely many 
			places $v$. 
	\end{enumerate}
\end{thm}

	The first statement says that the arithmetic conclusion about convergence to a root at infinitely many places is equivalent to a geometric conclusion about the dynamics of a certain fixed point for the Newton map. The statement itself requires only the language of global height fields, while its proof uses a Diophantine approximation technique specific to number fields. Our first goal in the present paper is to show that Theorem~\ref{Thm: FV Main}(i) is really a statement about global height fields by using a result of the second author \cite{Towsley}, which in turn is a consequence of Runge's method. 
	
\begin{thm}
\label{Thm: Arithmetic/Dynamical Exceptional}
	Let $K$ be a global height field, and let $(x_n)$ be the Newton approximation sequence associated to a point $x_0 \in K$ and a squarefree polynomial $f \in K[x]$ with $\deg(f) \geq 2$. Then the sequence $(x_n)$ converges $v$-adically to a root $\alpha$ of $f$ for infinitely many places $v$ of $K$ if and only if $\alpha$ is not an exceptional fixed point for the dynamical system~$N_f$. 
\end{thm}

\begin{rmk}
	The condition that $f$ be squarefree implies that $f'$ is not identically zero, and hence that $N_f$ is well defined. Moreover, if $K$ has characteristic $p > 0$ and $g$ is a function of $x^p$, say $g = h \left( x^p \right)$, then $N_{f\cdot g} = N_f$, which means the Newton map ignores inseparable factors. For purposes of root finding, it is also reasonable to assume that $f$ is squarefree by replacing $f$ with $f / \gcd(f, f')$ if necessary. 
\end{rmk}

	The conclusion in Theorem~\ref{Thm: FV Main}(i) that $(x_n)$ converges for infinitely many places $v$ if $\deg(f) > 2$ is a consequence of a classification of polynomials $f$ that have a dynamically exceptional root $\alpha$ --- i.e., a root that is an exceptional fixed point for $N_f$. For global height fields of characteristic zero, this classification carries over verbatim;  with some modification of the argument in positive characteristic, one can still give a classification of polynomials $f$ for which \textit{all} of its roots are exceptional fixed points for $N_f$ (Proposition~\ref{prop:Dynamically_Exceptional}). This immediately yields the following corollary. 
		
\begin{cor}
	Let $K^{\mathrm{a}}$ be a fixed algebraic closure of $K$. With the notation of Theorem \ref{Thm: Arithmetic/Dynamical Exceptional}, the sequence $(x_n)$ converges $v$-adically to some root of $f$ for infinitely many places $v$ of $K$ unless one of the following is true:
	\begin{itemize}
		\item $f$ is quadratic;
		\item $K$ has characteristic $p > 0$ 
			and there exist $A, B, C \in K^{\mathrm{a}}$ and an integer $r \geq 1$ 
			such that  $Af(Bx + C) = x^{p^r} - x$; or
		\item $K$ has characteristic $p > 0$ 
			and there exist $\lambda, A, B, C \in K^{\mathrm{a}}$ and an integer $r \geq 1$ 
			such that  $Af(Bx+C) = x^{p^r+1} - (\lambda+1)x^{p^r} + \lambda x$. 			 		
	\end{itemize}
\end{cor}

	We now discuss quantitative extensions of Theorem~\ref{Thm: FV Main}(ii). The first author and Voloch conjectured that something much stronger should be true in the case of number fields: the set of places $v$ of $K$ for which $(x_n)$ does not converge $v$-adically to any root of $f$ has density one when ordered by norm. They were only able to provide numerical and heuristic evidence to this effect. A clever technique using the Chebotarev density theorem was later developed by several authors to address a related question on collisions of orbits modulo primes \cite{BGHKST}. It implies that Newton's method fails to converge to some root for a \textit{positive} density of places of a number field \cite[Thm.~4.6]{BGHKST}. (It does not give an effective lower bound on the density.)  We are able to give a unified proof of positive density for any global field. 
	
\begin{thm}
\label{Thm: Positive divergence}
	Let $K$ be a finite extension of the rational field $\mathbb{Q}$ or of the rational function field $\mathbb{F}_p(T)$. Assume the setup of Theorem~\ref{Thm: Arithmetic/Dynamical Exceptional}. The set of places $v$ of $K$ for which the Newton approximation sequence $(x_n)$ diverges $v$-adically has positive lower density when ordered by norm. 
\end{thm}

\begin{rmk}
\label{Rmk: Silverman/Voloch}
	Silverman and Voloch proved a preliminary form of the above theorem for a number field $K$ or an arbitrary function field $K/k$ of transcendence degree~1. They showed that the set of places $v$ of $K$ for which the Newton approximation sequence diverges is infinite \cite[Remark~10]{Silverman_Voloch}. Surprisingly, they were able to show that the result continues to hold in the function field setting \textit{even when the map is purely inseparable}, and we model the purely inseparable case of our proof on theirs. 
\end{rmk}

	We now update the main conjecture of \cite{FV} to include the case of function fields:
	
\begin{conj}[Newton Approximation Fails for $100\%$ of Places] 
	Let $K$ be a finite extension of the rational field $\mathbb{Q}$ or of the rational function field $\mathbb{F}_p(T)$. Let $(x_n)$ be the Newton approximation sequence associated to a squarefree polynomial $f \in K[x]$ of degree at least~2 and a point $x_0 \in K$. Define $C(K,f, x_0)$ to be the set of places $v$ of $K$ for which $(x_n)$ fails to converge $v$-adically to some root of $f$. Then the set of places $C(K,f,x_0)$ has natural density one when ordered by norm. 
\end{conj}

	For other global height fields $K$, we expect the set of places of divergence to be ``large,'' but there is some issue in clarifying what ``large'' should mean. 
We are able to give a satisfying answer in the case where $K/k$ is the function field of a curve $Y$ defined over a global height field $k$. Here ``large'' may be quantified in terms of a height function on $Y$. 
	
\begin{thm}
\label{Thm: Large fields}
	Let $k$ be a global height field, and let $K/k$ be a function field of transcendence degree one. Fix a nonsingular connected curve $Y/k$ such that $K = k(Y)$, and fix a height function $h_Y$ on $Y(k^{\mathrm{a}})$ associated to an ample divisor. Let $(x_n)$ be the Newton approximation sequence associated to a squarefree polynomial $f \in K[x]$ of degree at least~2 and a point $x_0 \in K$. Then there exists a positive real number $M = M(f, x_0, h_Y)$ such that $(x_n)$ fails to converge in $K_y$ for all $y \in Y(k^{\mathrm{a}})$ satisfying $h_Y(y) \geq M$.
\end{thm}	


More generally, one can place a fundamental cardinality restriction on the set of places at which Newton approximation may detect a root of a polynomial. 

\begin{prop}
\label{Prop: Countably many}
	Let $K$ be a global height field, and let  $(x_n)$ be the Newton approximation sequence associated to a squarefree polynomial $f \in K[x]$ of degree at least~2 and a point $x_0 \in K$. The sequence $(x_n)$ converges $v$-adically to a root of $f$ for at most a countable number of places $v$ of $K$.
\end{prop}

	The proof is trivial. Pass to a finite extension of $K$ if necessary so that $f$ splits completely. For each of the countably many pairs $(n, \alpha) \in \mathbb{N} \times \mathrm{Zeros}(f)$, the inequality $|x_n - \alpha|_v < 1$ may only be satisfied for finitely many places $v$ of $K$. 
	
	Morally, the above proposition holds because there is a finitely generated subfield of $K$ over which $f$ and $x_0$ are defined, and a finitely generated field admits only countably many places. Note that, unlike $\mathbb{Q}$ or $\FF_p(T)$, the places of an arbitrary finitely generated field do not come with any sort of canonical ordering with which to measure density. With these observations in mind, we ask the following:

\begin{question}
	Is there a meaningful way to assert that Newton approximation fails
	for a ``large'' set of places of an arbitrary finitely generated global height field?
\end{question}
		
	To close the introduction, we describe the contents of this paper. Section~\ref{Sec: Geometry} describes a number of geometric properties of the correspondence $f \mapsto N_f$ between polynomials and Newton maps. This includes the following dynamical characterization: Up to conjugation, Newton maps are precisely the rational functions of degree~$d \geq 2$ that have at least $d$ critical fixed points. The results in Section~\ref{Sec: Geometry} are well known in characteristic zero, so our goal is to compare and contrast the behavior  of this correspondence in characteristic zero with that in positive characteristic. We prove Theorem~\ref{Thm: Arithmetic/Dynamical Exceptional} in Section~\ref{Sec: Positive result}, and the proof of Theorems~\ref{Thm: Positive divergence} and~~\ref{Thm: Large fields} will occupy Section~\ref{Sec: Negative result}.

\medskip

\noindent \textbf{Acknowledgments.} The authors would like to thank Tom Tucker for recommending that they embark on this collaboration, and for  suggesting the use of Call/Silverman specialization in the proof of Theorem~\ref{Thm: Large fields}. They also thank the anonymous referee for the observation that led to our formulation of Proposition~\ref{Prop: Countably many}, and  Asher Auel for useful comments on an earlier draft of the manuscript.


\section{Geometry of the Newton map}
\label{Sec: Geometry}
	
	Throughout this section, we fix a separably closed field $L$ of characteristic $p \geq 0$. We begin by describing separability properties of the Newton map $N_f$ in terms of the polynomial $f$. Then we discuss the fixed points of the Newton map and use them to give a dynamical characterization of rational functions that are conjugate to a Newton map. In the final subsection we characterize those squarefree polynomials $f$ for which every root is a totally ramified fixed point of the Newton map. 
	

\subsection{Inseparability}


\begin{defn}\label{Separable}
 A rational function $\vp \in L(x)$ is \textbf{inseparable} if $L$ has positive characteristic and $\vp$ can be written as $\vp \left( x \right) = \psi \left( x^p \right)$ for some $\psi \in L(x)$; $\varphi$ is said to be \textbf{separable} otherwise. By the Hurwitz formula, $\vp$ is inseparable if and only if its formal derivative is identically zero. We say that $\varphi$ is \textbf{purely inseparable} if $\vp \left(x \right) = \eta \left(x^{p^r} \right)$ with $r > 0$ and $\eta$ fractional linear. 
\end{defn}

	It will be important for our arithmetic questions to know which polynomials $f$ give rise to purely inseparable Newton maps~$N_f$.

\begin{rmk}
	A nonconstant rational function $\varphi \in L(x)$ is separable (resp. inseparable, purely inseparable) if and only if the field extension $L(x) / L(\varphi)$ is separable (resp. inseparable, purely inseparable) in the sense of classical field theory. 
\end{rmk}

\begin{prop}\label{Inseparable_Criterion}
Suppose that $L$ has positive characteristic, and let $f  \in L\left[x \right]$ be a squarefree polynomial of degree at least~2. 
 	\begin{enumerate}
		\item The Newton Map $N_f$ is inseparable if and only if 
			 $f \left( x \right) = x g \left( x^p \right) + h \left( x^p \right)$ 
			 for polynomials $g,h \in L \left[ x \right]$.
		
		\item The Newton map $N_f$ is purely inseparable if and only if 
			$f \left( x \right) = x g \left( x^{p^r} \right) + h \left( x^{p^r} \right)$ 
			 for polynomials $g,h \in L \left[ x \right]$ of degree at most~1 and a positive integer~$r$. 
	\end{enumerate}
\end{prop}

\begin{proof}
 Since $f $ is squarefree, its derivative cannot vanish identically, and $N_f$ is well defined.
 
 If $f \left(x  \right) = x g \left( x^p \right) + h \left( x^p \right)$, then $N_f \left( x \right) = - \frac{h\left(x^p \right)}{g \left(x^p \right)}$, which is inseparable.

 Conversely, assume that $N_f$ is inseparable. By definition, $N_f \left(x \right) = \frac{x f^\prime \left( x \right) - f \left( x \right)}{f^\prime \left( x \right)}$,
 and the numerator and denominator have no factor in common since $f$ is squarefree. It follows that $f^\prime \left( x \right) = g \left( x^p \right)$ for some $g \in L \left[ x\right]$. 
 Integration yields that $f \left( x \right) = x g \left( x^p \right) + h \left( x^p \right)$ for some~$h$. This completes the proof of (i). 

	For (ii), an immediate calculation shows that $N_f$ is purely inseparable if $f$ has the given form. Conversely, if $N_f$ is purely inseparable, then part (i) shows that $f(x) = x g \left(x^{p^r}\right) + h\left(x^{p^s}\right)$ for some polynomials $g, h$ and some integers $r,s > 0$. We may assume that $r$ and $s$ are chosen to be maximal. Then 	
		$N_f(x) = - \frac{h\left(x^{p^s}\right)}{g\left(x^{p^r}\right)}$.
As $N_f$ is purely inseparable, we must have $r = s$ and $h / g$ is fractional linear.
\end{proof}

	
\subsection{Degree and fixed points}
Typically, a polynomial of degree~$d$ will give rise to a Newton map of degree~$d$. However, in positive characteristic, the Newton map will have degree~$d-1$ when $d \equiv 1 \pmod p$. This difference has only a slight effect on the basic properties of the correspondence $f \mapsto N_f$. 
	
\begin{prop}\label{Newton_Degree}
	Let $f \in L[x]$ be a squarefree polynomial of degree $d \geq 2$. 
	\begin{enumerate}
		\item The Newton map has degree
			\[
				\deg(N_f) = \begin{cases}
					d & \text{if $d \neq 1$ in $L$} \\
					d - 1 & \text{if $d = 1$ in $L$}.
					\end{cases}
			\]

		\item Let $\alpha_1, \ldots, \alpha_d$ be the roots of $f$. Then the fixed points of the Newton map are
			\[
				\mathrm{Fix}(N_f) = \begin{cases}
					\{\alpha_1, \ldots, \alpha_d, \infty\} 
						& \text{if $d \neq 1$ in $L$} \\
					\{\alpha_1, \ldots, \alpha_d\} 
						& \text{if $d = 1$ in $L$}.	
					\end{cases}					
			\]

		\item Each of the roots $\alpha_i$ of $f$ is a critical fixed point for $N_f$.
			If $d \neq 1$ in $L$, then $\infty$ has fixed point multiplier 
			$\frac{d}{d-1}$. In particular, if $d = 0$ or~$1$ in $L$, then all of the 
			fixed points of $N_f$ are critical. 
	\end{enumerate}
\end{prop}

\begin{proof}
The squarefree hypothesis shows there is no common factor shared between the numerator and denominator of $N_f(x) = \left(xf' - f \right) / f'$. 

For part (i), observe that the leading terms of the numerator and denominator of $N_f$ are $(d - 1)x^{d}$ and $dx^{d-1}$, respectively. We are finished unless $d = 1$ in $L$ and $N_f = 0$. But then we find that $f(x) = xg(x^p) + h(x^p)$ for some polynomials $g,h$ that are at most linear (Proposition~\ref{Inseparable_Criterion}). Computing $N_f$ shows that $h = 0$, and the squarefree hypothesis on $f$ forces $g$ to be constant. As we assumed $\deg(f) \geq 2$, this is a contradiction. 

For part (ii), it is clear that the $\alpha_i$ are fixed points of $N_f$. If $d \neq 1$ in $L$, then the previous paragraph shows that the numerator of $N_f$ has larger degree than the denominator, so infinity is a fixed point. A rational function of degree $\delta$ has $\delta+1$ fixed points, so we have found them all.  

For (iii), we pick a root $\alpha$ of $f$ and write $f(x) = (x - \alpha)g(x)$ for some polynomial $g$. Now
	\[
		N_f(x) - \alpha = (x- \alpha) - \frac{f(x)}{f'(x)} =  (x-\alpha)^2 \frac{g'(x)}{g(x) + (x-\alpha)g'(x)}.
	\]
Since $f$ is squarefree, we find $g(\alpha) \neq 0$, and $N_f(x)$ is ramified at $\alpha$. 

	If $d \neq 1$ in $L$, we can write 	
	\[
		N_f(x) = \frac{(d-1)x^d + \cdots}{dx^{d-1} + \cdots},
	\]
where we are ignoring lower order terms. We compute the fixed point multiplier at $\infty$ by a change of coordinate and ignoring higher order terms:
	\[		
		1 / N_f(1/x) = \frac{dx + \cdots}{(d-1) + \cdots}
		=\frac{d}{d-1} x + \cdots  \qedhere
	\]
\end{proof}


\begin{cor}
\label{Cor: No deg 1 Newton maps}
	If $d = 1$ in $L$, then there is no Newton map of degree~$d$. 
\end{cor}

	The following result gives a dynamical description of Newton maps. It is certainly well known in the case of fields of characteristic zero; see for example \cite{Nishizawa_Fujimura}. The only difference that arises in positive characteristic is that it is possible to have rational functions all of whose fixed points are critical. Such a function is necessarily conjugate to a Newton map. 

\begin{prop}
\label{Prop: Newton locus is dynamical}
	Fix an integer $d \geq 2$, and let $\phi \in L(x)$ be a rational function of degree~$d$. The following are equivalent:
\begin{enumerate}
	\item $\phi$ is conjugate to a Newton map $N_f$ for some squarefree 
		polynomial $f \in L[x]$;
	\item $\phi$ has at least $d$ critical fixed points; and
	\item If $d \neq 0$ in $L$ (resp. $d = 0$ in $L$), then $\phi$ has exactly $d$ (resp. exactly $d+1$) critical fixed points. 
\end{enumerate}
\end{prop}

\begin{proof}
		(i) $\Rightarrow$ (ii). Evident from Proposition~\ref{Newton_Degree}.
	
	(ii) $\Rightarrow$ (i). Suppose that $\varphi$ has at least $d$ critical fixed points, and write $\varphi = A / B$ with $A$ and $B$ coprime polynomials. If one of the fixed points of $\varphi$ is not critical, we may conjugate it to $\infty$ without loss of generality. Otherwise, let us conjugate any of the critical fixed points to $\infty$. In particular, we may assume that $\deg(B) < \deg(A)$. 
	
	Set $f(x) = xB(x) - A(x)$; the roots of $f$ are exactly the finite fixed points of $\varphi$. We claim that $N_f = \varphi$. First, we show that $A' = xB'$. The finite critical points of $\varphi$ are the roots of 
	\[
		BA' - AB' = B(A' - xB') + B'(xB - A ). 
	\]
Since every finite fixed point is also a critical point, $xB - A \mid B(A' - xB')$. As $A$ and $B$ are coprime, we conclude that $xB - A \mid A' - xB'$. But now $\deg(xB - A) = d$, while $\deg(A' - xB') \leq d - 1$. Hence, $A' = xB'$, as desired. This relation implies that $f' = xB' - A' + B = B$, and 
	\[
		\gcd(f, f') = \gcd(xB - A, B) = \gcd(A,B) = 1.
	\]
That is, $f$ is squarefree. A direct computation shows that $N_f = A / B$, as desired. 

	(iii) $\Rightarrow$ (ii). Evident. 
	
	(ii) $\Rightarrow$ (iii). Write $\lambda_0, \ldots, \lambda_d$ for the multipliers at the $d+1$ (distinct) fixed points of $\phi$, and let us suppose that $\lambda_1 = \cdots = \lambda_d = 0$. The fixed point index formula shows that
	\[
		1 = \sum \frac{1}{1 - \lambda_i} = d + \frac{1}{1 - \lambda_0}. 
	\]
We see that $d = 0$ in $L$ if and only if $\lambda_0 = 0$. 
\end{proof}


\subsection{Dynamically exceptional roots}

\begin{defn}
\label{Def: Dynamical Exceptional}
	A root $\alpha$ of $f$ will be called  \textbf{dynamically exceptional} if it is exceptional for the Newton map; i.e., its backward orbit  $\{\gamma \in \PP^1(L): N_f^n(\gamma) = \alpha \text{ some $n \geq 0$}\}$ is finite. Equivalently, since roots of $f$ are fixed by $N_f$, $\alpha$ is dynamically exceptional if and only if it is totally ramified for $N_f$. It follows from the Hurwitz formula that if $N_f$ is not purely inseparable, then it has at most two exceptional points. 
\end{defn}

	The following lemma will allow us to move the roots of $f$ and adjust it to be monic without affecting the dynamical properties of the Newton map. 

\begin{lem}[{\cite[Prop.~2.4]{FV}}]
\label{Lem: Dyn Equiv}
	Let $f, g \in L[x]$ be polynomials of degree~$d$ related by the formula $g(x) = Af(Bx + C)$, where $A,B,C \in L$ and $AB \neq 0$. Set $\sigma(x) = Bx + C$. Then $N_g = \sigma^{-1} \circ N_f \circ \sigma$. 
\end{lem}

\begin{prop}\label{prop:Dynamically_Exceptional}
	Let $f \in L[x]$ be a squarefree polynomial of degree $d \geq 2$. 
	\begin{enumerate}
		\item If $d = 2$, then every root of $f$ is dynamically exceptional.

		\item If $d > 2$ and $d \neq 1$ in $L$, then every root of $f$ is dynamically exceptional 
			if and only if $L$ has positive characteristic and there exist $A, B, C \in L$ and
			an integer $r \geq 1$ such that $Af(Bx + C) = x^{p^r} - x$. 

		\item If $d > 2$ and $d = 1$ in $L$, then every root of $f$ is dynamically exceptional 
			if and only if $L$ has positive characteristic and there exist 
			$\lambda, A, B, C \in L$ and an integer $r \geq 1$ such that 
			$Af(Bx + C) = x^{p^r+1} - (\lambda+1)x^{p^r} + \lambda x$. 			
	\end{enumerate}
\end{prop}

\begin{proof}
	Suppose first that $f$ is quadratic, and that $\alpha, \beta$ are the roots of $f$. Then both $\alpha$ and $\beta$ are critical fixed points for $N_f$. A critical fixed point for a quadratic map is totally ramified, so that both $\alpha$ and $\beta$ are dynamically exceptional for $f$. 
	
	Now assume that $d = \deg(f) > 2$. 	It was shown in \cite[Cor.~1.2]{FV} that when $L$ has characteristic zero, a polynomial has at most one dynamically exceptional root. For the remainder of the proof, we may therefore assume that $L$ has characteristic $p > 0$. 
	
	Write $D$ for the degree of $N_f$. Suppose that $\alpha, \beta$ are distinct roots of $f$ that are totally ramified fixed points for $N_f$. 	By applying Lemma~\ref{Lem: Dyn Equiv}, we may assume without loss that $f$ is monic and that $\alpha = 0$ and $\beta = 1$. By conjugating $\beta$ to $\infty$ while fixing 0, we see that $N_f$ is conjugate to $\phi(x) = x^D$. At least one other fixed point of $N_f$ is totally ramified, so the same must be true of $\phi$. The fixed points of $\phi$ are $0, \infty$, and the $(D-1)^{\mathrm{st}}$ roots of unity. If $\zeta$ is such a root of unity, we find that
	\[
		\phi(x) - \zeta = x^D - \zeta = (x - \zeta)^D
	\]
if and only if $D = q$ is a power of the characteristic and $\zeta \in \FF_q$. For $a \in L \smallsetminus \{0\}$, write $\sigma(x) = ax / (x-1)$ for a fractional linear transformation that fixes $0$ and moves $1$ to $\infty$. There exists $a \in L$ such that
	\begin{equation}
	\label{Eq: Newton special form}
		N_f(x) = \sigma^{-1} \circ \phi \circ \sigma(x) = \frac{a^{q-1}x^q}{a^{q-1}x^q - (x-1)^q}
			= \frac{a^{q-1}x^q}{a^{q-1} x^q - x^q + 1}. 
	\end{equation}
	
	If $d \neq 1$ in $L$, then $N_f$ fixes infinity, which means $a^{q-1} = 1$. Thus $N_f(x) = x^q$. Proposition~\ref{Inseparable_Criterion} shows that $f(x) = xg(x^q) + h(x^q)$ for polynomials $g,h$ that are at most linear; computing $N_f$ directly shows that $g = -1$ and $h = x$. (Recall that we have assumed $f$ is monic.)
	
	If instead $d = 1$ in $L$, then $N_f$ does not fix infinity, so that $a^{q-1} \neq 1$. Replacing $a^{q-1}$ with $1 + 1 / \lambda$ in  \eqref{Eq: Newton special form} shows that 
	\[
		N_f(x) = \frac{(\lambda+1)x^q}{x^q + \lambda}.
	\]
As in the previous case,  $f(x) = xg(x^q) + h(x^q)$ for some linear polynomials $g, h$; computing $N_f$ directly shows that $g = x + \lambda$ and $h = - (\lambda + 1)x$.
	
	Conversely, if $f$ is of either of the two special forms in the statement, then $N_f$ is purely inseparable (Proposition~\ref{Inseparable_Criterion}), which implies that all of its fixed points are totally ramified. 
\end{proof}


\section{Success of Newton's method over global height fields}
\label{Sec: Positive result}

\subsection{Global height fields}
\label{Sec: Height fields}
	
	We follow the notation and conventions in \cite[\S1]{Bombieri}. 
	
\begin{defn}
	A field $K$ will be called a \textbf{global height field} if it is equipped with a collection $M_K$ of inequivalent nontrivial absolute values $|\cdot|_v$ such that the following two conditions hold for each $\alpha \in K \smallsetminus \{0\}$:
	\begin{itemize}
		\item (Adelic Property) The set $\{v \in M_K : |\alpha|_v \neq 1\}$ is finite, and
		\item (Product Formula) 	$\displaystyle \prod_{v \in M_K} |\alpha|_v = 1$.
	\end{itemize}
\end{defn}

Examples are given by number fields and function fields $k(X)$, where $X$ is a geometrically irreducible normal projective variety defined over a field~$k$.\footnote{Fix a nontrivial numerically effective divisor $H$ on $X$. Each prime divisor $Y$ on $X$ gives rise to an absolute value $| \cdot |_Y$ defined by $|f|_Y = \exp( -n_Y \mathrm{ord}_Y(f))$ for $f \in k(X)$, where $n_Y = Y \cdot H^{\dim(X) - 1}$.} See also \cite{Corvaja_Zannier_Infinite_Extensions} and \cite[\S2]{LangDiophantine}  for further interesting examples. Every finite extension $K' / K$ may be endowed with the structure of a global height field where each absolute value on $K'$ restricts to one in $M_K$, up to equivalence. We write $M_{K'}$ for the corresponding set of absolute values on $K'$ normalized as in \cite[\S1.3.6, 1.3.12]{Bombieri}. (The normalization is unimportant for the results presented here.)
	
	Let $K$ be a global height field with a fixed algebraic closure $K^{\mathrm{a}}$. The absolute logarithmic height function $h: \PP^1(K^{\mathrm{a}}) \to \mathbb{R}_{\geq 0}$ is defined as follows. Let $\alpha = (\alpha_0 : \alpha_1) \in \PP^1(K^{\mathrm{a}})$, and let $K' / K$ be a finite extension containing the coordinates of $\alpha$. Then we define
	\[
		h(\alpha) =  \sum_{w \in M_{K'}} \log \max\{|\alpha_0|_w, |\alpha_1|_w\}. 
	\]
By the product formula, the sum is independent of the choice of homogeneous coordinates for $\alpha$, and the normalizations of the elements of $M_{K'}$ are arranged so that $h(\alpha)$ is independent of the choice of extension $K'/K$ containing $\alpha$. 


\subsection{Arithmetically exceptional roots}

	Throughout this section, $K$ will denote a fixed global height field. 

\begin{defn}\label{Def: Arithmetic Exceptional}
Let $(x_n)$ be a Newton approximation sequence associated to some squarefree polynomial $f \in K[x]$ and some $x_0 \in K$. (Recall that we assume $(x_n)$ is not eventually periodic.) A root $\a \in K^{\mathrm{a}}$ of $f$ will be called \textbf{arithmetically exceptional} for $f$ if the Newton approximation sequence $\left( x_n \right)$ converges to $\a$ in  $K_v$ for only finitely many places $v \in M_K$. 
\end{defn}

	Definitions~\ref{Def: Dynamical Exceptional} and~\ref{Def: Arithmetic Exceptional} allow us to restate  Theorem~\ref{Thm: Arithmetic/Dynamical Exceptional} as an equivalence between arithmetic and dynamical exceptionality:

\begin{thm}
\label{THM: Success}
	Let $(x_n)$ be a Newton approximation sequence associated to some squarefree polynomial $f \in K[x]$ and a point $x_0 \in K$. Then a root of $f$ is arithmetically exceptional if and only if it is dynamically exceptional. 
\end{thm}

The proof is a consequence of the following result of the second author. Note that although it was originally stated for function fields over a finite constant field, its proof uses only the properties of a global height field. 

\begin{thm}[{\cite[Cor.~3.3]{Towsley}}]
\label{Thm: Towsley1}
	Let $\varphi \in K(x)$ be a rational function of degree at least~2. Let $\beta, \gamma \in \PP^1(K)$. Suppose that $\gamma$ is periodic and not exceptional for $\varphi$. If $\beta$ is not preperiodic, then there exist infinitely many places $v$ of $K$ such that $\left|\varphi^n(\beta) - \gamma\right|_v < 1$  for some $n\geq 0$. 
\end{thm}

\begin{proof}[Proof of Theorems~\ref{Thm: Arithmetic/Dynamical Exceptional} and~\ref{THM: Success}]
	By passing to a finite extension if necessary, we may assume that all of the roots of $f$ lie in $K$. Let $\alpha \in K$ be such a root. Write $f$ as $f \left(x \right) = \left( x- \a \right) g(x)$ for some polynomial $g \in K[x]$ that does not vanish at $\alpha$. 
 Define $S_\a$ to be the finite set of places $v \in M_K$ such that at least one of the following is true:
 	\begin{itemize}
		\item The leading coefficient of $f'$ has $v$-adic absolute value different from~1;
		\item $|\alpha - \beta|_v \neq 1$ for each root $\beta$ of $f'$; 
		\item $|\beta|_v \leq 1$ for each root $\beta$ of $f'$; and
		\item $|\alpha|_v \leq 1$. 
	\end{itemize}

Let $v \in M_K \smallsetminus S_\a$, and suppose that $n \geq 0$ is such that $|x_n - \alpha|_v < 1$. By definition, 
	\begin{align*}
		x_{n+1} - \alpha = N \left( x_n \right)  - \alpha
			&= x_n - \alpha - \frac{\left(x_n  - \a \right) g \left( x_n \right)}{\left( x_n - \a \right)g^\prime \left( x_n \right) + g \left( x_n \right)} \\
			 &=   \left( x_n - \a \right)^2  \cdot \frac{g^\prime \left( x_n \right)}
			 	{f'(x_n) } 
	 \end{align*}
 As $v \notin S_\a$, the denominator of this last expression has $v$-adic absolute value~1. Hence, $|x_{n+1} - \a|_v \leq |x_n - \a|_v^2$. By induction, $|x_{n+\ell} - \a|_v \leq |x_n - \a|_v^{2^\ell}$ for every $\ell \geq 0$. It follows that the Newton approximation sequence $\left( x_n \right)$ converges to $\a$ in the $v$-adic topology for $v \in M_K \smallsetminus S_\a$ if and only if $x_n$ lies in the strict open disk of radius~1 about $\alpha$ for some $n \geq 0$. 
 
	Now suppose that $\alpha$ is not dynamically exceptional for $f$. Theorem~\ref{Thm: Towsley1} may be applied in our setting with $\beta = x_0$ and $\gamma = \alpha$ to show there are infinitely many places $v$ for which  $|x_n - \a|_v < 1 $ for some $n$. So $\a$ is not arithmetically exceptional by the preceding paragraph. 

	Finally, suppose that $\a$ is dynamically exceptional. Then 
	\[
		N_f(x) - \alpha = c \frac{(x - \alpha)^{\deg(N_f)}}{f'(x)}
	\]
for some nonzero $c \in K$. Setting $x = x_n$, we see that 
	\begin{equation}
	\label{Eq: Not to infinity}
		x_{n+1} - \alpha = c \frac{(x_n - \alpha)^{\deg(N_f)}}{f'(x_n)}.
	\end{equation}
Let $S_\a'$ be $S_\a$ along with the finitely many other $v$ for which $|c|_v \neq 1$. Take $v \not\in S_\a'$, and suppose that $|x_{n+1} - \alpha|_v < 1$ for some $n \geq 0$. Then either $|f'(x_n)|_v > 1$ or $|x_n - \a|_v < 1$. If the former is true, then our hypothesis on $v$ shows that $|x_n - \beta|_v > 1$ for some root $\beta$ of $f'$, in which case $|x_n|_v > 1$.  But \eqref{Eq: Not to infinity} shows that $|x_{n+1} - \alpha|_v = |x_n|_v^{\deg(N_f) - \deg(f')} \geq 1$, a contradiction. So we are forced to conclude that $|x_n - \alpha|_v < 1$ whenever $|x_{n+1} - \alpha|_v < 1$. By induction, $|x_n - \alpha|_v < 1$ if and only if $|x_0 - \alpha|_v < 1$. As $(x_n)$ is not eventually periodic, there are only finitely many places $v$ for which this last inequality can hold. Hence, $\alpha$ is arithmetically exceptional. 
\end{proof}


\section{Failure of Newton's method}
\label{Sec: Negative result}

\subsection{Global fields}
	Throughout this subsection, we let $K$ be a global field --- that is, a finite extension of the field of rational numbers $\mathbb{Q}$ or of the function field $\FF_p(T)$ for some rational prime $p$. Let $K^{\mathrm{a}}$ be a fixed algebraic closure of $K$. We write $\oo_K$ for the integral closure of $\mathbb{Z}$ or $\FF_p[T]$ inside $K$. In the function field setting, this is equivalent to saying that $K$ is the function field of a smooth connected projective curve $C/\FF_p$ endowed with a finite morphism $C \to \PP^1_{\FF_p}$, and $\oo_K$ is the ring of rational functions on $C$ that are regular away from the points lying above $\infty \in \PP^1_{\FF_p}$. 

	Write $M_K^{0}$ for the set of finite places of $K$; each place $v \in M_K^{0}$ may be canonically identified with a unique prime ideal $\pp_v \in \oo_K$, and conversely. We will interchange $v$ and $\pp_v$ as needed without further comment.  The norm of $v \in M_K^{0}$, denoted $\N v$, is the order of the residue field associated to $\pp_v$: $\N v = \# \oo_K / \pp_v$. The (natural) lower density of a set of places $S \subset M_K^0$ is defined to be
	\[
		\underline{\delta}(S) = \liminf_{X \to \infty} \frac{\#\{v \in S : \N v \leq X\}}{\#\{v \in M_K^0 : \N v \leq X\}}.
	\] 

\begin{rmk}
	Only finitely many places of $K$ are ignored by focusing on $M_K^0$, so this omission has no effect on any of the results in this section. 
\end{rmk}

 	The proof of Theorem~\ref{Thm: Positive divergence} will follow the general strategy of \cite{FV}, but in place of the primitive prime divisor result from \cite{FG}, we use a result of Benedetto et al. for number fields  \cite{BGHKST} and the corresponding result of the second author for function fields \cite{Towsley}. For the statement, we recall that if $K$ is a function field of characteristic~$p$ and $\varphi \in K(x)$ is an inseparable rational function, then it may be written as $\varphi(x) = \psi(x^{p^r})$ for some $r \geq 1$ and some \textit{separable} rational function $\psi$. The degree of  $\psi$ is called the \textbf{separable degree} of $\varphi$. Equivalently, the separable degree of $\varphi$ is the number of $K^{\mathrm{a}}$-rational solutions to the equation $\varphi(x) = \beta$ for a generic choice of $\beta$. To say that the separable degree of $\varphi$ is greater than 1 is to say that $\varphi$ does not induce a bijection on $\PP^1(K^{\mathrm{a}})$. 
	
\begin{thm}[{\cite[Thm.~1.1]{BGHKST}, \cite[Thm.~4]{Towsley}}]
\label{Thm: Towsley2}
	Let $K$ be a global field, and let $\varphi \in K(x)$ be a rational function of degree at least~2. If $K$ is a function field, suppose further that $\varphi$ has separable degree at least~2. Let $\gamma \in \PP^1(K)$ be a point with infinite $\varphi$-orbit, and let $\mathcal{B} \subset \PP^1(K)$ be a finite set of preperiodic points for $\varphi$. Then the set of places $\mathcal{P} \subset M_K^0$ such that $\varphi^n(\gamma) \not\equiv \beta \pmod \pp$ for any $n \geq 0$, any $\beta \in \mathcal{B}$, and any $\pp \in \mathcal{P}$, has positive lower density.
\end{thm}

\begin{proof}[Proof of Theorem~\ref{Thm: Positive divergence}]
	We wish to show that the set of places $v$ of the global field $K$ at which $(x_n)$ fails to converge $v$-adically to a root of $f$ has positive lower density. Let us suppose first that $N_f$ is not purely inseparable. 
Let $\mathcal{B}$ denote the set of fixed points of $N_f$. (See Proposition~\ref{Newton_Degree}.)
A continuity argument shows that if $\left( x_n \right)$ converges in $\PP^1(F_\p)$ for a prime $\p \in M_K^0$, then it converges to an element of $\mathcal{B}$. Applying Theorem~\ref{Thm: Towsley2} with $\gamma = x_0$ shows that the sequence $(x_n)$ fails to $v$-adically approach $\mathcal{B}$ for a set of places $v$ of $K$ with positive lower density. 
 
	Now we suppose that $N_f(x)$ is purely inseparable; hence, $K$ necessarily has positive characteristic~$p$. After conjugating if necessary, we may replace $N_f$ with $ x^q$ for $q$ some power of~$p$. It suffices to show that $(x_n)$ converges $v$-adically to an element of $\PP^1(\FF_q)$ for at most finitely many places $v$. Let $\alpha \in \FF_q$. Write $y_0 = x_0 - \alpha$. A suitable induction hypothesis shows that $x_n = \alpha + y_0^{q^n}$. Then for any place $v$, we see that $|x_n - \alpha|_v = |y_0|^{q^n}_v$, so that $(x_n)$ converges $v$-adically to $\alpha$ only for those places $v$ for which $|y_0|_v < 1$. The set of such places $v$ is finite. A similar argument applies in the case $\alpha = \infty$ if we set $y_0 = 1 / x_0$. 
\end{proof}


\subsection{Function fields over global height fields}
	
	Before beginning the proof of Theorem~\ref{Thm: Large fields}, we recall that if $k$ is a global height field with height function $h_k$, and if  $\phi \in k(x)$ is a rational function of degree at least~2, there is a canonical height function $\hat h_{k, \phi}: \PP^1(k^{\mathrm{a}}) \to \mathbb{R}_{\geq 0}$ characterized by the following two properties:
	\begin{itemize}
		\item $\hat h_{k, \phi}\left(\phi(\alpha)\right) = \deg(\phi) \cdot \hat h_{k, \phi}(\alpha)$
			for any $\alpha \in \PP^1(k^{\mathrm{a}})$.
		\item The difference $|\hat h_{k, \phi}(\alpha) - h_k(\alpha)|$ for $\alpha \in \PP^1(k^{\mathrm{a}})$ may be bounded by a constant that does not depend on $\alpha$.  
	\end{itemize}
	
	Recall also that if $K/k$ is a function field as in the statement of Theorem~\ref{Thm: Large fields}, we say that a rational function $\phi \in K(x)$ is isotrivial if there is $\sigma \in \PGL_2(K^{\mathrm{a}})$ such that $\sigma \circ \phi \circ \sigma^{-1} \in k^{\mathrm{a}}(x)$. That is, up to a coordinate change and a finite extension of $K$, the function $\varphi$ is defined over the constant field. We prove the following fact about canonical heights for isotrivial maps:

 	\begin{lem}
        \label{LEM: Isotrival Heights}
	  Let $K/k$ be a function field over a global height field $k$, let $\vp \in K \left( x \right)$ be an isotrival rational function, and let $x_0 \in K$. Suppose that the isotriviality of $\vp$ is witnessed by $\sigma \in PGL_2 \left( K^a \right)$; that is, $\sigma \circ \vp \circ \sigma^{-1} \in k \left( x \right)$. The canonical height $\hat{h}_{K,\vp}\left( x_0 \right)\rightarrow \infty$ as $n \rightarrow \infty$ unless $x_0$ is preperiodic or $\sigma(x_0) \in \PP^1(k^a)$. 
        \end{lem}

	\begin{proof}
	If $x_0$ is preperiodic, then $\hat{h}_{K,\vp} \left( x_0 \right) = 0$, so we assume that $x_0$ is not preperiodic.
	Let $\psi := \sigma \circ \vp \circ \sigma^{-1} \in k \left( x \right)$ and define $\xi_n = \sigma \left( x_n \right)$. Observe that $\xi_{n+1} = \psi(\xi_n)$ for every $n \geq 0$. If $\xi_0 \in k^{\mathrm{a}}$, then $h_K(\xi_n) = 0$ for all $n$. This implies $\hat h_{K, \psi}(\xi_n) = 0$ for all $n$. Since canonical heights are conjugation invariant, we conclude that $\hat h_{K, \vp}(x_n) = 0$ for all $n$.
	
	Finally, we assume that $\xi_0$ is not defined over (an extension of) the constant field. There necessarily exists a place $z$ of $K$ such that $|\xi_0|_z > 1$. Write $\psi = g/h$ for some coprime polynomials $g, h \in k[x]$. The coefficients of $g$ and $h$ all have $z$-adic absolute value~1; hence, $\max\{|g(\xi_0)|_z, |h(\xi_0)|_z\} = |\xi_0|_z^d$, where $d = \deg(\psi) = \max\{\deg(g), \deg(h)\}$. Since $\xi_n = g(\xi_{n-1}) / h(\xi_{n-1})$ for all $n \geq 1$, we can proceed by induction to show that 
	\[
		\max\{|g(\xi_{n-1})|_z, |h(\xi_{n-1})|_z\} = |\xi_0|_z^{d^{n}}. 
	\]
It follows that
	\[
		h_K(\xi_{n})  = \sum_{y \in M_K} \log \max\{|g(\xi_{n-1})|_y, |h(\xi_{n-1})|_y\}
			\geq  \log |\xi_0|_z^{d^{n}}. 
	\]
	Hence, $h_K(\xi_n) \to \infty$ as $n \to \infty$. Since the canonical height $\hat h_{K, \psi}$ differs from the usual height $h_K$ by a bounded amount, $\hat h_{K, \psi}(\xi_n) \to \infty$, and  conjugation invariance of canonical heights shows that $\hat{h}_{K, \vp}(x_n) \to \infty$ as well. 
\end{proof}

	Using this lemma, a theorem of Baker, and the specialization result of Call and Silverman, we now prove Theorem~\ref{Thm: Large fields}.

\begin{proof}[Proof of Theorem \ref{Thm: Large fields}]
        Suppose first that the Newton map $N_f$ is not isotrivial. Since $x_0$ is not preperiodic for $N_f$, we apply Baker's theorem \cite[Cor.~1.8]{Baker} to conclude that the canonical height $\hat h_{K,N_f}(x_0)$ is positive. Write $x_{n,y}$ for the specialization of $x_n$ at $y \in Y(k^{\mathrm{a}})$, and similarly write $N_{f,y}$ for the induced rational function with coefficients in $k^{\mathrm{a}}$. Write $\deg(Y)$ for the degree of the ample divisor associated to $h_Y$. Then Call/Silverman specialization \cite[Thm.~4.1]{Call-Silverman} shows that 
	\[
		\hat h_{k,N_{f,y}}(x_{0,y}) \sim \hat h_{K, N_f}(x_0) \cdot \frac{h_Y(y)}{\deg(Y)} \text{ \ as $h_Y(y) \to \infty$}.
	\]
 In particular, if $h_Y(y)$ is sufficiently large, then $\hat h_{k,N_{f,y}}(x_{0,y}) > 0$. It follows that $x_{0,y}$ is not preperiodic for the specialized map $N_{f,y}$ when $y$ has large height, and in particular this means that it cannot ever encounter a fixed point of the map $N_{f,y}$. This is sufficient to prove that $(x_n)$ does not converge $y$-adically to a root of $f$.

	Now suppose that $N_f$ is isotrivial, and choose $\sigma \in \PGL_2 \left( K^{\mathrm{a}} \right)$ such that $\psi := \sigma \circ \varphi \circ \sigma^{-1} \in k^{\mathrm{a}}(x)$. Set $\xi_n = \sigma \left(x_n \right)$. Note that $\xi_n = \psi(\xi_{n-1})$ for each $n \geq 1$. If $\xi_0 \notin k^a$ then we apply Lemma \ref{LEM: Isotrival Heights} to conclude that $\hat{h}_{K, \psi} \left( \xi_n \right)>0$ for all $n \geq n_0$. If $n_0 > 0$, then replace $x_0$ with $x_{n_0}$. The strategy in the previous paragraph applies again to show that $(x_n)$ does not converge $y$-adically to a root of $f$ as long as $h_Y(y)$ is sufficiently large. 

 	Finally, suppose that $N_f$ is isotrivial as in the last paragraph, and that $\xi_0 = \sigma(x_0) \in k^a$. Without loss of generality, we will assume it lies in $\PP^1(k)$. The entire sequence $(\xi_n)$ must also be defined over $k$. But $(\xi_n)$ is therefore discrete in $\PP^1(K_y)$ for every place $y$, and so it cannot converge unless it is eventually constant. This contradicts the hypothesis that $(x_n)$ is not eventually periodic. We conclude that Newton approximation fails to converge for all places $y$ in this case. 
	\end{proof}


\bibliographystyle{plain}
\bibliography{bib}
\end{document}